\documentclass[11pt,a4paper]{article}

\usepackage{epsf,epsfig,amsfonts,amsgen,amsmath,amstext,amsbsy,amsopn,amsthm,cases,listings,color
}
\usepackage{ebezier,eepic}
\usepackage{color}
\usepackage{multirow}
\usepackage{epstopdf}
\usepackage{graphicx}
\usepackage{pgf,tikz}
\usepackage{mathrsfs}
\usepackage[marginal]{footmisc}
\usepackage{enumitem}
\usepackage{appendix}

\usepackage{pgf,tikz,pgfplots}
\usepackage{subfigure}
\usepackage{mathrsfs}
\usetikzlibrary{arrows}

\allowdisplaybreaks[1]

\definecolor{uuuuuu}{rgb}{0.27,0.27,0.27}
\definecolor{sqsqsq}{rgb}{0.1255,0.1255,0.1255}

\newtheorem{dfn}{Definition} [section]
\newtheorem{thm}[dfn]{Theorem}
\newtheorem{lemma}[dfn]{Lemma}

\newtheorem{coro}[dfn]{Corollary}

\newtheorem{claim}[dfn]{Claim}

\newtheorem{fact}[dfn]{Fact}
\newenvironment{pf}{\noindent{\bf Proof}}

\begin{document}
\title{\bf\Large Tight bounds for Katona's shadow intersection theorem}

\date{\today}

\author{Xizhi Liu\thanks{Department of Mathematics, Statistics, and Computer Science, University of Illinois, Chicago, IL, 60607 USA. Email: xliu246@uic.edu.
Research partially supported by NSF award DMS-1763317.}
\and
Dhruv Mubayi\thanks{Department of Mathematics, Statistics, and Computer Science, University of Illinois, Chicago, IL, 60607 USA. Email: mubayi@uic.edu.
Research partially supported by NSF award DMS-1763317.}
}
\maketitle
%\footnote{footnote}
%%%%%%%%%%%%%%%%%%%%%%%%%%%%%%%%%%%%%%%%%%%%%%%%%
\begin{abstract}
A fundamental result in extremal set theory is Katona's shadow intersection theorem,
which extends the Kruskal-Katona theorem by giving a lower bound on the size of the shadow of
an intersecting family of $k$-sets in terms of its size. We improve this classical result and
a related result of Ahlswede, Aydinian and Khachatrian by proving  tight bounds for   families that can be quite small.
For example, when $k=3$ our result is sharp for all families with $n$ points and at least $3n-7$ triples.

Katona's theorem was extended by Frankl to families with matching number $s$.
We improve Frankl's result  by giving tight bounds for large $n$.
\end{abstract}
%%%%%%%%%%%%%%%%%%%%%%%%%%%%%%%%%%%%%%%%%%%%%%%%%

\section{Introduction}
Let $n\ge k \ge \ell\ge 1$.
Given a family $\mathcal{H}\subset \binom{[n]}{k}$ the $\ell$-th shadow of $\mathcal{H}$ is
\begin{align}
\partial_{\ell}\mathcal{H} = \left\{A\in \binom{[n]}{k-\ell}: \exists B\in \mathcal{H} \text{ such that }A\subset B \right\}. \notag
\end{align}
When $\ell = 1$ we write $\partial\mathcal{H}$ and call $\partial\mathcal{H}$ the shadow of $\mathcal{H}$.
The colex order on ${[n] \choose k}$  is defined as follows:
\begin{align}
A \prec B \text{ iff } \max\{(A\setminus B)\cup (B\setminus A)\}\in B. \notag
\end{align}
Write $L_{m}\mathcal{H}$ to denote the set of the first $m$ elements of $\mathcal{H} \subset {[n] \choose k}$ in the colex order. When $\mathcal{H} = {[n] \choose k}$, we abuse notation by simply writing $L_{m}\binom{[n]}{k}$.

The celebrated Kruskal-Katona theorem states that the families in ${[n] \choose k}$ with a fixed number of sets and minimum shadow size are initial elements of the colex order.

%shows that for all $n \ge k \ge 1$ and $1\le \ell \le k$ the minimum size of
%$\partial_{\ell}\mathcal{H}$ overall $\mathcal{H}\subset\binom{[n]}{k}$ is achieved by %$L_{|\mathcal{H}|}\binom{[n]}{k}$.

\begin{thm}[Kruskal-Katona~\cite{KA68,KR63}]\label{thm-Kruskal-Katona}
For $n\ge k >\ell  \ge 1$ and $\mathcal{H}\subset \binom{[n]}{k}$ with $|\mathcal{H}| = m$,
 $$|\partial_{\ell}\mathcal{H}|\ge \left |\partial_{\ell}L_{m}\binom{[n]}{k}\right|.$$
\end{thm}
%%%%%%%%%%%%%%%%%%%%%%%%%%%%%%%%%%%%%%%%%%%%%%%%%%%%%%
\subsection{Katona's shadow intersection Theorem}
The Kruskal-Katona theorem was extended to families with additional properties.
One such result is due to Katona \cite{KA64} about $t$-intersecting families, which are families in which every two sets have at least $t$ common elements.

\begin{thm}[Katona~\cite{KA64}]\label{thm-Katona-t-intersecting}
Let $n\ge k > t \ge \ell \ge 1$.
If $\mathcal{H} \subset \binom{[n]}{k}$ is $t$-intersecting, then
$$|\partial_{\ell}\mathcal{H}|\ge \frac{\binom{2k-t}{k-\ell}}{\binom{2k-t}{k}} |\mathcal{H}|.$$
\end{thm}

The only case of equality in Theorem \ref{thm-Katona-t-intersecting} is when $n=2k-t$ and $\mathcal{H} \cong \binom{[2k-t]}{k}$ (see \cite{AAK04}).

Theorem~\ref{thm-Katona-t-intersecting} is a foundational result in extremal set theory with many applications.
Its first application was to prove a conjecture of Erd\H os-Ko-Rado on the maximum size of a  $t$-intersecting family in $2^{[n]}$.
It was used to obtain short new proofs for several classical results.
For example, Frankl-F\"{u}redi \cite{FF12} used it to give a short proof for the Erd\H{o}s-Ko-Rado theorem,
and Frankl-Tokushige \cite{FT92} used it to obtain a short proof for the Hilton-Milner theorem.
It also has many applications to Sperner families and other types of intersection problems~\cite{BF19,CG81,FF84,FF86,G80,MB18,S92,Z89}.

This paper is concerned with improving the bounds in Theorem~\ref{thm-Katona-t-intersecting} and related results about shadows of families with certain properties. In many cases the bounds we prove are best possible.

Our first result improves Theorem~\ref{thm-Katona-t-intersecting} for intersecting families  (the case $t=1$) and applies to all  $n > 2k$. It is convenient to define  the family
\begin{align}
EKR(n,k) = \left\{A\in {[n]\choose k}:  1 \in A\right\}. \notag
\end{align}

\begin{thm}\label{thm-shadow-intersecting-all-n}
Let $n > 2k\ge 6$  and $1\le \ell < k$.
Suppose that $\mathcal{H} \subset \binom{[n]}{k}$ is intersecting and
\begin{align}
|\mathcal{H}| = m > m(n,k) =
\begin{cases}
3n-8, & \text{ if } k=3, \\
\binom{n-1}{k-1}-\binom{n-k}{k-1}+\binom{n-k-2}{k-3}+2, & \text{ if } k \ge 4.
\end{cases}\notag
\end{align}
Then $|\partial_{\ell}\mathcal{H}| \ge |\partial_{\ell} L_{m}EKR(n,k)|$.
%and equality holds iff $\mathcal{H} \cong L_{m}EM(n,k)$.
In particular, if for some $x\in \mathbb{R}$
\begin{align}
|\mathcal{H}| = \binom{x-1}{k-1} > m(n,k)
\end{align}
then $|\partial_{\ell}\mathcal{H}| \ge \binom{x}{k-\ell}$.
\end{thm}
%
%
%The lower bound for $|\mathcal{H}|$ in Theorem \ref{thm-shadow-intersecting-all-n} for $k=3$
%is tight because one can easily check that $|\partial L_{14}HM(n,3,1)| = |\partial L_{14}EM(n,3)| = 20$
%and $|\partial L_{13}HM(n,3,1)| < |\partial L_{13}EM(n,3)|$.

{\bf Remarks.}
\begin{itemize}
    \item
\ For $k=3$ and $m=3n-8$, the inequality $|\partial\mathcal{H}|<|\partial L_{m}EKR(n,k)|$ is possible
(see Fact~\ref{fact1.7} with $t=1$), so Theorem~\ref{thm-shadow-intersecting-all-n} is best possible in this sense. In fact, when $k=3$ one can compute the sharp lower bound for $|\partial{\mathcal H}|$ {\em for all} intersecting families  $\mathcal{H}$ using our proof method but we do not carry out all these details.

\item For fixed $k>3$ and  $n\rightarrow \infty$, we will lower the value of $m(n,k)$ from $(k-1+o(1)){n \choose k-2}$ to   $(3+o(1)){n \choose k-2}$ in Theorem~\ref{thm-EM(n,k,s)-n^k-2} and the constant $3$ will be shown to be  tight.
\end{itemize}
\medskip

Ahlswede, Aydinian, and Khachatrian \cite{AAK04} considered large $t$-intersecting families on $\mathbb{N}$. Let $\binom{\mathbb{N}}{k}$ denote the collection of all $k$-subsets of $\mathbb{N}$
and let $$EM(\mathbb{N},k,s,t) = \left\{A \in \binom{\mathbb{N}}{k}: |A\cap [s]|\ge t\right\}.$$

\begin{thm}[Ahlswede, Aydinian, and Khachatrian~\cite{AAK04}]\label{thm-AAK04}
Let $\mathcal{H}\subset \binom{\mathbb{N}}{k}$ be a $t$-intersecting family.
\begin{itemize}
\item For $1 \le \ell \le t < k$, there exists $m_{1}(k,t,\ell) \in \mathbb{N}$ such that if $|\mathcal{H}|=m\ge m_{1}(k,t,\ell)$,
then $|\partial_{\ell}\mathcal{H}|\ge |\partial_{\ell}L_{m}EM(\mathbb{N},k,2k-2-t,k-1)|$.
\item For $1 \le t < \ell < k$, there exists $m_{2}(k,t,\ell) \in \mathbb{N}$ such that if $|\mathcal{H}|=m\ge m_{2}(k,t,\ell)$,
then $|\partial_{\ell}\mathcal{H}|\ge |\partial_{\ell}L_{m}EM(\mathbb{N},k,t,t)|$.
\end{itemize}
\end{thm}

For $0 \le t \le \min \{k,s\}$, let
\begin{align}
EM(n,k,s,t) = \left\{A\subset {[n]\choose k}:  |A\cap [s]|\ge t \right\}, \notag
\end{align}
and set $EM(n,k,s,t) = \emptyset$ if $t > \min\{k,s\}$, and  $EM(n,k,s,t) = \binom{[n]}{k}$ if $n \le s$.

For every $m \le \binom{n-t}{k-t}$ we have $L_{m}EM(n,k,t,t) = L_{m}EM(\mathbb{N},k,t,t)$.
Therefore, Theorem \ref{thm-AAK04} implies the following result.

\begin{coro}\label{coro-AAK04}
Let $1 \le t < \ell < k$ and $\mathcal{H}\subset \binom{[n]}{k}$ be a $t$-intersecting family with $|\mathcal{H}| = m > m_{2}(k,t,\ell)$.
Then $|\partial_{\ell}\mathcal{H}|\ge |\partial_{\ell}L_{m}EM(n,k,t,t)|$.
\end{coro}

However, for the case $\ell \le t$ we  show that  the smallest possible size of the $\ell$-th shadow of
large $t$-intersecting families on $[n]$ is different than the formula in Theorem~\ref{thm-AAK04}. Let
\begin{align}
AK(n,k,t) = \left\{ A \in \binom{[n]}{k}: [t] \subset A \text{ and } [t+1,k+1] \cap A \neq\emptyset \right\} \cup
            \left( \bigcup_{i\in [t]}\left\{[k+1]\setminus\{i\} \right\} \right). \notag
\end{align}
%Let
%\begin{align}
%EM(n,k,s,t) = \left\{A\subset {[n]\choose k}:  |A\cap [s]|\ge t \right\}, \notag
%\end{align}
%and set $EM(n,k,s,t) = \emptyset$ if $t > \min\{k,s\}$, and  $EM(n,k,s,t) = \binom{[n]}{k}$ if $t\le 0$.
Notice that $AK(n,k,t)$ and $EM(n,k,t+2,t+1)$ are both $t$-intersecting,
$$|AK(n,k,t)| \sim (k-t+1)\binom{n}{k-t-1},$$
$$|EM(n,k,t+2,t+1)| \sim (t+2)\binom{n}{k-t-1}.$$
Our next result is a finite version of Theorem~\ref{thm-AAK04}.
\medskip
\begin{thm}\label{thm-shadow-t-intersecting-all-n}
Let $t\ge 1, k\ge 3, 1 \le \ell < k$, and $n> (t+1)(k-t+1)$.
Suppose that $\mathcal{H} \subset \binom{[n]}{k}$ is $t$-intersecting and
\begin{align}
|\mathcal{H}| = m > m(n,k,t) =
\begin{cases}
\max\left\{|AK(n,k,t)|,|EM(n,k,t+2,t+1)|\right\}, & \text{ if } t < \frac{k-1}{2}, \\
|EM(n,k,t+2,t+1)|, & \text{ if } t \ge \frac{k-1}{2}.
\end{cases}\notag
\end{align}
Then $|\partial_{\ell}\mathcal{H}|\ge |\partial_{\ell} L_{m}EM(n,k,t,t)|$.
%, and equality holds iff $\mathcal{H} \cong L_{m}EM(n,k,t,t)$.
In particular, if
\begin{align}
|\mathcal{H}| = \binom{x-t}{k-t} > m(n,k,t)
\end{align}
for some $x\in \mathbb{R}$.
Then $|\partial_{\ell}\mathcal{H}|\ge \sum_{i= t-\ell}^{k-\ell}\binom{t}{i}\binom{x-t}{k-\ell-i}$.
For  $1\le \ell \le t$ the  value of $m(n,k,t)$ is tight for $t \ge \frac{k-1}{2}$
and is tight up to a constant multiplicative factor independent of $n$ for $t < \frac{k-1}{2}$.
\end{thm}
%
%\begin{rmk}
%\begin{itemize}
%
%\end{itemize}
%\end{rmk}
{\bf Remarks.}
\begin{itemize}
       \item
Theorem \ref{thm-shadow-t-intersecting-all-n} implies that for a $t$-intersecting family $\mathcal{H} \subset \binom{[n]}{k}$
with large size, $$\frac{|\partial_{\ell}\mathcal{H}|}{|\mathcal{H}|} > \binom{t}{\ell} \ge  \frac{\binom{2k-t}{k-\ell}}{\binom{2k-t}{k}}$$ for $1 \le \ell \le t$ with equality in the second inequality iff $\ell = t$. Hence our bound is better than that in Theorem \ref{thm-Katona-t-intersecting} (as expected since our bound is best possible).

\item For $t < \frac{k-1}{2}$ we will show in the last section that the lower bound for
$|\mathcal{H}|$ in Theorem \ref{thm-shadow-t-intersecting-all-n} can be improved slightly.
\end{itemize}

\subsection{Frankl's theorem}
The matching number of $\mathcal{H}$, denoted by $\nu(\mathcal{H})$,
is the maximum number of pairwise disjoint edges in $\mathcal{H}$.
Notice that $\nu(EM(n,k,s,1)) \le s$ with equality iff $n \ge ks$ and
$$|EM(n,k,s,1)| = \binom{n}{k}-\binom{n-s}{k} \sim s \binom{n}{k-1} \qquad \hbox{($n \rightarrow \infty$)}.$$
The Erd\H{o}s matching conjecture \cite{ER65} says that for all $n\ge (s+1)k-1$, if $\mathcal{H} \subset \binom{[n]}{k}$
and $\nu(\mathcal{H})\le s$, then
\begin{align} \label{eqEMC}
|\mathcal{H}| \le \max\left\{\binom{(s+1)k-1}{k}, \binom{n}{k}-\binom{n-s}{k}\right\}.
\end{align}

When $s=1$, (\ref{eqEMC}) follows from the  Erd\H{o}s-Ko-Rado theorem \cite{EKR61}.

\begin{thm}[Erd\H{o}s-Ko-Rado~\cite{EKR61}]\label{thm-EKR}
Let $k\ge 2$ and $n\ge 2k$, $\mathcal{H}\subset \binom{[n]}{k}$ be an intersecting family.
Then $\mathcal{H} \le \binom{n-1}{k-1}$ and when $n>2k$ equality holds iff $\mathcal{H} \cong EKR(n,k)$.
\end{thm}

The Erd\H{o}s matching conjecture is still open  and
the current record on this conjecture is due to Frankl \cite{FR13}.

\begin{thm}[Frankl~\cite{FR13}]\label{thm-Frankl-Matching-Conj}
Let $k\ge 2$ and $n\ge (2s+1)k - s$, $\mathcal{H}\subset \binom{[n]}{k}$ and $\nu(\mathcal{H}) \le s$.
Then $\mathcal{H} \le \binom{n}{k} - \binom{n-s}{k}$ with equality iff $\mathcal{H} \cong EM(n,k,s,1)$.
\end{thm}

If we take $t=1$ in Theorem \ref{thm-Katona-t-intersecting},
then every intersecting family
$\mathcal{H} \subset \binom{[n]}{k}$ satisfies $|\partial\mathcal{H}|\ge |\mathcal{H}|$. Frankl generalized this as follows.

\begin{thm} [Frankl~\cite{FR91, FR13}] \label{franklnu=s}
Let $n \ge k \ge 2$ and $\mathcal{H} \subset {[n] \choose k}$. If $\nu(\mathcal{H})=s\ge 1$,
then $$|\partial\mathcal{H}|\ge \frac{|\mathcal{H}|}{s}$$
with equality iff $\mathcal{H}\cong \binom{[(s+1)k-1]}{k}$.
\end{thm}

Theorem~\ref{franklnu=s} is a crucial tool in the proof of Theorem~\ref{thm-Frankl-Matching-Conj} and any improvement in Theorem~\ref{franklnu=s} for small values of $n$  could lead to a corresponding improvement in Theorem~\ref{thm-Frankl-Matching-Conj}. Our final result provides such an improvement (for large $n$) that is sharp if $|\mathcal{H}|$ is large.
%%%%%%%%%%%%%%%%%%%%%%%%%%%%%%%%%%%%%%%%%%%%%%%

\begin{thm}\label{thm-EM(n,k,s)-n^k-2}
For every $k\ge 3$ and every $s\ge 1$ there exists  $m(n,k,s)$  such that the following holds as $n \rightarrow \infty$.
Suppose that $\mathcal{H} \subset \binom{[n]}{k}$ satisfies $\nu(\mathcal{H})\le s$ and
\begin{align}\label{hlowerc}
|\mathcal{H}| = m > m(n,k,s) =
\begin{cases}
(3+o(1)){n \choose k-2} & \text{ if } s=1 \\
(\binom{2s+1}{2}+o(1)){n \choose k-2} & \text{ if } k=3 \\
(k\binom{s+1}{2} + o(1)){n \choose k-2} & \text{ if } k \ge 4, s \ge 2.
\end{cases}
\end{align}
Then $$|\partial\mathcal{H}| \ge |\partial L_{m}EM(n,k,s,1)|.$$
%and equality holds iff $\mathcal{H} \cong L_{m}EM(n,k,s,1)$.
In particular, if $|\mathcal{H}| = \binom{x}{k} - \binom{x-s}{k} > m(n,k,s)$
for some $x \in \mathbb{R}$, then $|\partial\mathcal{H}| \ge \binom{x}{k-1}$.
\end{thm}

The constraint $\nu(\mathcal{H})\le s$ above imposes the bound $|\mathcal{H}| = O(n^{k-1})$, so the point of Theorem~\ref{thm-EM(n,k,s)-n^k-2} is that it applies to $|\mathcal{H}| \ge (c(k,s)+o(1)){n \choose k-2}$ where $c(k,s)$ is obtained from (\ref{hlowerc});  this is a lower order of magnitude  than $n^{k-1}$. In fact, as we will show below, the order of magnitude $n^{k-2}$ is best possible for such a result and even the constant $c(k,s)$ is tight if $s=1$ or $k=3$,
 and is tight up to a constant factor for all other $(s,k)$.

%
%The proof of Theorem \ref{thm-EM(n,k,s)-n^k-2} also gives the following technically simpler result.
%
%\begin{coro}\label{coro-EM(n,k,s)-n^k-2-simple}
%For every $k\ge 3$ and every $s\ge 1$ there exists constants $c = c(k,s)$ and $n_0 = n_0(k,s)$ such that the following holds for all $n\ge n_0$.
%Suppose that $\mathcal{H} \subset \binom{[n]}{k}$ satisfies $\nu(\mathcal{H})\le s$
%and $|\mathcal{H}| = \binom{x}{k} - \binom{x-s}{k} \ge c\binom{n}{k-2}$ for some $x \in \mathbb{R}$.
%Then $|\partial\mathcal{H}| \ge \binom{x}{k-1}$.
%\end{coro}

%Our proof of Theorem \ref{thm-EM(n,k,s)-n^k-2} gives
%\begin{align}\label{equ-c(k,s)}
%c(k,s) =
%\begin{cases}
%3 & \text{ if } s=1 \\
%\binom{2s+1}{2} & \text{ if } k=3 \\
%\binom{s+1}{2}k & \text{ if } k \ge 4, s \ge 2.
%\end{cases}
%\end{align}

%Let us now show that
%the lower bound for $c(k,s)$ (given by (\ref{hlowerc}))
%in Theorem \ref{thm-EM(n,k,s)-n^k-2} is tight for $(k,s)$ if $s=1$ or $k=3$,
 %and is tight up to a constant factor for all other $(s,k)$.

Let $\mathcal{G} = EM(n,k,2s+1,2)$ and
$m = |\mathcal{G}| \sim \binom{2s+1}{2}\binom{n}{k-2}$ and
let $x \in \mathbb{R}$ such that $\binom{x}{k} -\binom{x-s}{k} = m$.
Since $\binom{x}{k} -\binom{x-s}{k} \sim s \binom{x}{k-1}$, $x = \Theta(n^{\frac{k-2}{k-1}})$.
Notice that
\begin{align}
s|\partial \mathcal{G}|- m
& = s \sum_{i=1}^{k-1}\binom{2s+1}{i}\binom{n-2s-1}{k-1-i} - \sum_{i=2}^{k}\binom{2s+1}{i}\binom{n-2s-1}{k-i} \notag\\
& = \Theta(n^{k-3}), \notag
\end{align}
and
\begin{align}
s|\partial L_{m}EM(n,k,s,1)|- m
& \ge s \binom{x}{k-1} -\left( \binom{x}{k} -\binom{x-s}{k}\right) \notag\\
& = \Theta(x^{k-2})= \Theta(n^{\frac{(k-2)^2}{k-1}}). \notag
\end{align}
Since $\frac{(k-2)^2}{k-1} > k-3$, $|\partial L_{m}EM(n,k,s,1)| > |\partial EM(n,k,2s+1,2)|$ for sufficiently large $n$.
Therefore, we obtain the following result.

\begin{fact}
For every $k \ge 3$ and sufficiently large $n$
there exists $\mathcal{G} \subset \binom{[n]}{k}$ with $\nu(\mathcal{G}) =s$ and $|\mathcal{G}| = (1+o_{n}(1))\binom{2s+1}{2}\binom{n}{k-2}$
such that $|\partial\mathcal{G}| < |\partial L_{|\mathcal{G}|}EM(n,k,s,1)|$.
\end{fact}

It would be interesting to determine the minimum value of $c(k,s)$ such that the conclusion in Theorem \ref{thm-EM(n,k,s)-n^k-2} holds for
all $|\mathcal{H}| > c(k,s)\binom{n}{k-2}$ and sufficiently large $n$.

%%%%%%%%%%%%%%%%%%%%%%%%%%%%%%%%%%%%%%%%%%%%%%%%%%%%%%%%%%
\section{Proofs}
\subsection{Extension of the $k$-cascade representation}
In this section, we prove an extension of the well-known $k$-cascade representation of a number.
The $k$-cascade representation  plays an important role in the Kruskal-Katona theorem and the extension that we prove plays an analogous role for our theorems.
As a convention, let $\binom{a}{b} = 0$ if $b< 0$ or $a<b$, and let $\binom{a}{0} = 1$ for all $a \ge 0$.

For an $r$-graph $\mathcal{H}$ and a vertex set $S$ that is disjoint from $V(\mathcal{H})$ define
\begin{align}
\mathcal{H}+S = \left\{A \cup S: A \in \mathcal{H}  \right\}.   \notag
\end{align}
For every $i \in \mathbb{N}$ let $\widehat{i} = i+1$.

\begin{lemma}\label{lemma-decompose-LmEM(n,k,s,t}
Let $n \ge k \ge t \ge 0$ and $s \ge t \ge 0$. Then the following hold.
\begin{itemize}
\item[(a)] $|EM(n,k,s,t)| = \binom{n}{k} - \sum_{j=0}^{t-1}\binom{s}{j}\binom{n-s}{k-j}$.
\item[(b)] For every $1 \le m \le |EM(n,k,s,t)|$
there exist integers $a_k > a_{k-1}> \cdots > a_h \ge h \ge \max\{t,1\}$ such that
\begin{align}
L_{m}EM(n,k,s,t)= EM(a_{k},k,s,t) \cup \bigcup_{i=h}^{k-1}
\left( EM(a_{i},i,s,t)+\{\widehat{a}_{i+1}, \ldots, \widehat{a}_{k}\} \right)
\notag
\end{align}
\end{itemize}
\end{lemma}
%\left(EM(a_{k-1},k-1,s,t)+\{a_{k}+1\}\right) \notag\\
%& \quad \cup \cdots \cup \left( EM(a_{h},h,s,t)+\{a_{k}+1,\ldots,a_{h+1}+1\} \right).
\begin{proof}
$(a)$ is clear.
So let us consider $(b)$.

First, it follows from the definition that the colex order of $EM(n',k,s,t)$ is the initial segment of the colex on $EM(n,k,s,t)$ for all $n' < n$.
Let $\mathcal{F} = L_{m}EM(n,k,s,t)$.
Without loss of generality we may assume that $\mathcal{F} \neq EM(n',k,s,t)$ for all $n' \le n$
since otherwise we can let $h=k$ and $a_k = n'$ and we are done.
So there exists $a_k$ such that $EM(a_k,k,s,t) \subset \mathcal{F} \subset EM(a_{k}+1,k,s,t)$
and hence every set in $\mathcal{F} \setminus EM(a_k,k,s,t)$ contains $a_{k}+1$.
Therefore, $\mathcal{F} = EM(a_k,k,s,t) \cup \left(\mathcal{F}_{k}+\{\widehat{a}_k\}\right)$ for some $\mathcal{F}_{k} \subset EM(a_k,k-1,s,t)$.

Let $m' = |\mathcal{F}_{k}|$.
Then it follows from the definition of colex order that $\mathcal{F}_{k} = L_{m'}EM(a_{k},k-1,s,t)$.
So we can repeat the argument above to show that there exists $a_{k-1}$ such that
$\mathcal{F}_{k} = EM(a_{k-1},k-1,s,t) \cup \left(\mathcal{F}_{k-1}+\{\widehat{a}_{k-1}\}\right)$.
This means that
\begin{align}
\mathcal{F} = EM(a_k,k,s,t) \cup \left(EM(a_{k-1},k-1,s,t)+\{\widehat{a}_k\}\right) \cup \left(\mathcal{F}_{k-1}+\{\widehat{a}_k,\widehat{a}_{k-1}\}\right).\notag
\end{align}
Inductively, one will get a decomposition of $\mathcal{F}$ as in Lemma \ref{lemma-decompose-LmEM(n,k,s,t}.
\end{proof}

\begin{lemma}\label{lemma-k-cascade-resp}
Let $s \ge t \ge 0$ and $k \ge t$.
Then, for every integers $m\ge 1$,
there exists  a unique representation of $m$ in the form
\begin{align}
m = \sum_{i=h}^{k}\binom{a_i}{i} - \sum_{j=0}^{t-1} \binom{s}{j}\sum_{i=h}^{k}\binom{a_i-s}{i-j}, \notag
\end{align}
where $a_k > \cdots > a_h \ge h \ge \max\{t,1\}$ are integers.
%Moreover, if $h = t$, then $a_t \le s$.
\end{lemma}
\begin{proof}
If $t=0$, then this is just the $k$-cascade representation of $m$.
So we may assume that $t\ge 1$.
Let $n\in \mathbb{N}$ be sufficiently large such that $m \le |EM(n,k,s,t)|$.
Then the existence of such a representation follows from Lemma \ref{lemma-decompose-LmEM(n,k,s,t} since
\begin{align}
m
= |L_{m}EM(n,k,s,t)|
= \sum_{i=h}^{k}|EM(a_{i},i,s,t)|
& = \sum_{i=h}^{k}\left(\binom{a_{i}}{i}-\sum_{j=0}^{t-1}\binom{s}{j}\binom{a_i-s}{i-j}\right) \notag\\
& = \sum_{i=h}^{k}\binom{a_i}{i} - \sum_{j=0}^{t-1} \binom{s}{j}\sum_{i=h}^{k}\binom{a_i-s}{i-j}. \notag
\end{align}

Next, we prove the uniqueness of such representation of $m$.
Suppose that there exists $a_k > \cdots > a_h \ge h \ge t$ and $b_k > \cdots > b_{h'} \ge h' \ge t$
such that
\begin{align}\label{equation-m-k-cascade}
\sum_{i=h}^{k}\binom{a_i}{i} - \sum_{j=0}^{t-1}\binom{s}{j}\sum_{i=h}^{k}\binom{a_i-s}{i-j}
= m
= \sum_{i=h'}^{k}\binom{b_i}{i} - \sum_{j=0}^{t-1} \binom{s}{j}\sum_{i=h'}^{k}\binom{b_i-s}{i-j}.
\end{align}
Without loss of generality we may assume that $a_k \neq b_k$ since otherwise we can consider
$m' = m -\left(\binom{a_k}{i} - \sum_{j=0}^{t-1}\binom{s}{j}\binom{a_k-s}{i-j}\right)$ instead.
Let
\begin{align}
\mathcal{F}_{a}
= EM(a_{k},k,s,t) \cup \bigcup_{i=h}^{k-1} \left(EM(a_{i},i,s,t)+\{\widehat{a}_{i+1}, \ldots, \widehat{a}_k\}\right) \notag
\end{align}
and
\begin{align}
\mathcal{F}_{b}
= EM(b_{k},k,s,t) \cup \bigcup_{i=h'}^{k-1} \left(EM(b_{i},i,s,t)+\{\widehat{b}_{i+1}, \ldots, \widehat{b}_k\}\right). \notag
\end{align}
Then
\begin{align}
|\mathcal{F}_a| =
\sum_{i=h}^{k}\binom{a_i}{i} - \sum_{j=0}^{t-1}\binom{s}{j}\sum_{i=h}^{k}\binom{a_i-s}{i-j}, \notag
\end{align}
and
\begin{align}
|\mathcal{F}_b| =
\sum_{i=h'}^{k}\binom{b_i}{i} - \sum_{j=0}^{t-1} \binom{s}{j}\sum_{i=h'}^{k}\binom{b_i-s}{i-j}. \notag
\end{align}
Without loss of generality
we may assume that $a_k \ge b_k +1$.
However, notice that in this case $\mathcal{F}_b$ is a proper subset of $\mathcal{F}_a$, since every set of $\mathcal{F}_b$ has maximum element at most $b_k+1 \le a_k$. This contradicts (\ref{equation-m-k-cascade}).
%\begin{claim}\label{claim-a_k=b_k}
%$a_k = b_k$.
%\end{claim}
%\begin{proof}[Proof of Claim \ref{claim-a_k=b_k}]
%Suppose not.
%Without loss of generality we may assume that $a_k > b_k$.
%First we show that $a_k \le b_k +1$.
%Indeed, if $a_k > b_k +1$, then
%\begin{align}
%& \sum_{i=h'}^{k}\binom{b_i}{i} - \sum_{j=0}^{t-1}\binom{s}{j}\sum_{i=h'}^{k}\binom{b_i-s}{i-j} \notag\\
%& \le \sum_{i=t}^{k}\binom{b_k-(k-i)}{i} - \sum_{j=0}^{t-1} \binom{s}{j}\sum_{i=t}^{k}\binom{b_k-(k-i)-s}{i-j} \notag\\
%& \le \binom{b_k+1}{k}-\sum_{j=0}^{t-1}\binom{s}{j}\binom{b_k+1-s}{k-j} \notag \\
%& < \sum_{i=h}^{k}\binom{a_i}{i} - \sum_{j=0}^{t-1}\binom{s}{j}\sum_{i=h}^{k}\binom{a_i-s}{i-j}, \notag
%\end{align}
%a contradiction. Therefore, $a_k \le b_k +1$.
%
%Now suppose that $a_k = b_k +1$.
%Then the calculation above shows that we must have $h'=t=1$, $h=k$, and $b_{k-i} = b_{k}-i \ge s+1$ for $1\le i \le k$.
%In particular, $a_t \ge s+1$, which is a contradiction.
%Therefore, $a_k = b_k$.
%\end{proof}
%
%Inductively, one can show that $a_i = b_i$ for all $\max\{h,h'\} \le i \le k$, and hence $h=h'$ as well.
%This completes the proof of Lemma \ref{lemma-k-cascade-resp}.
\end{proof}

\subsection{Shifting}
For every $A \in \mathcal{H}$ and $1\le i < j \le n$ define
\begin{align}
S_{ij}(A) =
\begin{cases}
(A\setminus\{j\}) \cup \{i\}, & \text{ if } j\in A, i\not\in A, \text{ and } (A\setminus\{j\}) \cup \{i\} \not\in \mathcal{H}, \\
A, & \text{ otherwise}.
\end{cases}\notag
\end{align}
Let $S_{ij}(\mathcal{H}) = \{S_{ij}(A): A\in \mathcal{H}\}$ and call $\mathcal{H}$ shifted if $\mathcal{H} = S_{ij}(\mathcal{H})$
for all $1\le i < j \le n$.

\begin{fact}[see \cite{FR87}] \label{fact}
The following statements hold for all $\mathcal{H} \subset \binom{[n]}{k}$ and all $1\le i <j \le n$ and all $1\le t,\ell\le k-1$.
\begin{itemize}
\item $|\mathcal{H}| = |S_{ij}(\mathcal{H})|$.
\item $\partial_{\ell}S_{ij}(\mathcal{H}) \subset S_{ij}(\partial_{\ell}\mathcal{H})$ and in particular,
        $|S_{ij}(\partial_{\ell}\mathcal{H})| \ge |\partial_{\ell}S_{ij}(\mathcal{H})|$
\item $\nu(S_{ij}(\mathcal{H})) \le \nu(\mathcal{H})$.
\item If $\mathcal{H}$ is $t$-intersecting, then $S_{ij}(\mathcal{H})$ is also $t$-intersecting.
\end{itemize}
\end{fact}

\subsection{Main Lemma}
Fact~\ref{fact} shows that it suffices to consider shifted families in all proofs in this paper. The main technical statement in this work is Lemma~\ref{lemma-shadow-k-s-t-exact} below which is a generalization of the Kruskal-Katona theorem. For two families $\mathcal{H}_1$ and $\mathcal{H}_2$ we write $\mathcal{H}_1 \subset \mathcal{H}_2$
if $\mathcal{H}_1$ is isomorphic to a subgraph of $\mathcal{H}_2$.

Given a family $\mathcal{H}$, let $\mathcal{H}(1) = \{A\setminus \{1\}: 1\in A\in \mathcal{H}\}$
and $\mathcal{H}(\bar{1}) = \{A \in \mathcal{H}: 1\not\in A\}$.
It is easy to see that if $\mathcal{H}$ is shifted, then $\partial\mathcal{H}(\bar{1}) \subset \mathcal{H}(1)$
and hence $|\partial\mathcal{H}| = |\mathcal{H}(1)|+|\partial\mathcal{H}(1)|$.

\begin{lemma}\label{lemma-size-shadow-LmEM(n,k,s,t)}
Let $n \ge k \ge t \ge 0$ and $s \ge t \ge 0$.
Suppose that
\begin{align}
m =  \sum_{i=h}^{k}\binom{a_i}{i} - \sum_{j=0}^{t-1} \binom{s}{j}\sum_{i=h}^{k}\binom{a_i-s}{i-j} \notag
\end{align}
for integers $a_k > \cdots > a_{h} \ge \max\{t,1\}$.
Then for $1 \le \ell \le k-1$
\begin{align}
|\partial_{\ell} L_{m}EM(n,k,s,t)|
=  \sum_{i=h}^{k}\binom{a_i}{i-\ell} - \sum_{j=0}^{t-1-\ell} \binom{s}{j}\sum_{i=h}^{k}\binom{a_i-s}{i-\ell-j}.\notag
\end{align}
\end{lemma}
\begin{proof}
%Let us consider the case $\ell=1$ first.
%For an $r$-graph $\mathcal{H}$ and a vertex set $S$ that is disjoint from $V(\mathcal{H})$ define
%\begin{align}
%\mathcal{H}+S = \left\{A \cup S: A \in \mathcal{H}  \right\}.   \notag
%\end{align}
%Let $m$ be given as in Lemma \ref{lemma-size-shadow-LmEM(n,k,s,t)}.
%Then, it follows from the definition of colex order that
%\begin{align}
%L_{m}EM(n,k,s,t)
%& = EM(a_{k},k,s,t) \cup \left(EM(a_{k-1},k-1,s,t)+\{a_{k}+1\}\right) \notag\\
%& \quad \cup \cdots \cup \left( EM(a_{h},h,s,t)+\{a_{k}+1,\ldots,a_{h+1}+1\} \right).   \notag
%\end{align}
Fix $1 \le \ell \le k-1$. By Lemma \ref{lemma-decompose-LmEM(n,k,s,t},
\begin{align}
L_{m}EM(n,k,s,t)= EM(a_{k},k,s,t) \cup \bigcup_{i=h}^{k-1}
\left(EM(a_{i},i,s,t)+\{\widehat{a}_{i+1}, \ldots, \widehat{a}_{k}\}\right).
\notag
\end{align}
Notice that
\begin{align}
\partial  EM(a_k,k,s,t)
= EM(a_k,k-1,s,t-1),
\notag
\end{align}
and for every $h \le i \le k-1$ we have
\begin{align}
\partial \left(EM(a_i,i,s,t) + \{\widehat{a}_{i+1},\ldots,\widehat{a}_{k}\}\right)
=  \left(EM(a_i,i-1,s,t-1) + \{\widehat{a}_{i+1},\ldots,\widehat{a}_{k}\}\right) \cup  \notag\\
\bigcup_{j=i+1}^{k} \left(EM(a_i,i,s,t) + \{\widehat{a}_{i+1},\ldots,\widehat{a}_{k}\}\setminus \{\widehat{a}_{j}\}\right).
\notag
\end{align}
On the other hand, for all $h\le i<j \le k-2$ since $a_{j} > a_{i}$ ,
\begin{align}
EM(a_i,i,s,t) + \{\widehat{a}_{i+1},\ldots,\widehat{a}_{k}\}\setminus \{\widehat{a}_{j}\}
\subset EM(a_j,j-1,s,t-1) + \{\widehat{a}_{j+1},\ldots,\widehat{a}_{k}\}.
\notag
\end{align}
For all $h \le i \le k-1$ since $a_{k} > a_{i}$,
\begin{align}
EM(a_i,i,s,t) + \{\widehat{a}_{i+1},\ldots,\widehat{a}_{k}\}\setminus \{\widehat{a}_{k}\}
\subset EM(a_k,k-1,s,t-1).
\notag
\end{align}
Therefore,
\begin{align}
\partial L_{m}EM(n,k,s,t)
= \bigcup_{i=h}^{k} \left(EM(a_{i},i-1,s,t-1)+\{\widehat{a}_{i+1},\ldots,\widehat{a}_{k}\}\right), \notag
\end{align}
and inductively we obtain for all $1 \le \ell \le k-1$
\begin{align}
\partial_{\ell} L_{m}EM(n,k,s,t)
= \bigcup_{i=h}^{k} \left(EM(a_{i},i-\ell,s,t-\ell)+\{\widehat{a}_{i+1},\ldots,\widehat{a}_{k}\}\right). \notag
\end{align}
%that {\color{red} why is there an $\ell$  below on RHS but not on LHS}
%\begin{align}
%\partial L_{m}EM(n,k,s,t)
%= \bigcup_{i=h}^{k} EM(a_{i},i-\ell,s,t-1)+\{a_{i+1}+1, \ldots, a_{k}+1\} \notag
%\end{align}
%which implies {\color{red} need some more steps here to go to $\ell$ shadow -- can we first express the $\ell$ %shadow as a disjoint union of families, then add them up?}
Therefore,
\begin{align}
|\partial_{\ell} L_{m}EM(n,k,s,t)|
&= \sum_{i=h}^{k} |EM(a_i,i-\ell,s,t-\ell)| \notag \\
& =  \sum_{i=h}^{k}\left(\binom{a_i}{i-\ell}-\sum_{j=0}^{t-1-\ell}\binom{s}{j}\binom{a_i}{i-\ell-j}\right) \notag\\
& = \sum_{i=h}^{k}\binom{a_i}{i-\ell} - \sum_{j=0}^{t-1-\ell} \binom{s}{j}\sum_{i=h}^{k}\binom{a_i-s}{i-\ell-j}. \notag
\end{align}
\end{proof}

\begin{lemma}\label{lemma-shadow-k-s-t-exact}
Let $s\ge t \ge 0$.
If $\mathcal{H} \subset EM(n,k,s,t)$ and $|\mathcal{H}| = m$, then
$$|\partial\mathcal{H}|\ge |\partial L_{m}EM(n,k,s,t)|.$$
%Equality holds iff $\mathcal{H} \cong L_{m}EM(n,k,s,t)$.
%In particular,
%$$|\mathcal{H}|=\binom{x}{k} - \sum_{j=0}^{t-1}\binom{s}{j}\binom{x-s}{k-j} \quad
%\Longrightarrow  \quad |\partial\mathcal{H}| \ge \binom{x}{k-1} - \sum_{j=0}^{t-2}\binom{s}{j}\binom{x-s}{k-1-j}$$
%where $x\in \mathbb{R}$.
\end{lemma}
\begin{proof}
By Lemma \ref{lemma-k-cascade-resp}, there exists $a_k > \cdots > a_{h} \ge \max\{t,1\}$
such that
\begin{align}
m =  \sum_{i=h}^{k}\binom{a_i}{i} - \sum_{j=0}^{t-1} \binom{s}{j}\sum_{i=h}^{k}\binom{a_i-s}{i-j}. \notag
\end{align}
%On the other hand, it is easy to see that
%\begin{align}
%L_{m}EM(n,k,s,t) =  \sum_{i=h}^{k}\binom{a_i}{i} - \sum_{j=0}^{t-1}\binom{s}{j}\sum_{i=h}^{k}\binom{a_i-s}{i-j},\notag
%\end{align}
%and
%\begin{align}
%\partial L_{m}EM(n,k,s,t) =  \sum_{i=h}^{k}\binom{a_i}{i-1} - \sum_{j=0}^{t-2} \binom{s}{j}\sum_{i=h}^{k}\binom{a_i-s}{i-1-j}.\notag
%\end{align}
%Therefore,
Then, by Lemma \ref{lemma-size-shadow-LmEM(n,k,s,t)} it suffices to show that
%if for some integers $a_k> \cdots > a_h\ge h\ge \max\{t,1\}$
%\begin{align}
%|\mathcal{H}| = \sum_{i=h}^{k}\binom{a_i}{i} - \sum_{j=0}^{t-1}\binom{s}{j}\sum_{i=h}^{k}\binom{a_i-s}{i-j}, \notag
%\end{align}
%then
\begin{align}
|\partial\mathcal{H}| \ge \sum_{i=h}^{k}\binom{a_i}{i-1} - \sum_{j=0}^{t-2}\binom{s}{j}\sum_{i=h}^{k}\binom{a_i-s}{i-1-j}. \notag
\end{align}
We prove this statement by induction on $k,s,t$.
When $s=0$ or $k=1$ the statement is trivially true.
When $t=0$ the statement follows from the Kruskal-Katona theorem.
So we may assume that $s\ge t \ge 1$ and $k\ge 2$.

\begin{claim}\label{claim-H(1)-k-s-t}
$|\mathcal{H}(1)| \ge \sum_{i=h}^{k}\binom{a_i-1}{i-1} - \sum_{j=0}^{t-2}\binom{s-1}{j}\sum_{i=h}^{k}\binom{a_i-s}{i-1-j}$.
\end{claim}
\begin{proof}[Proof of Claim \ref{claim-H(1)-k-s-t}]
Suppose not.
Then
\begin{align}
|\mathcal{H}(\bar{1})|
& = |\mathcal{H}| - |\mathcal{H}(1)| \notag\\
& > \sum_{i=h}^{k}\binom{a_i}{i} - \sum_{j=0}^{t-1}\binom{s}{j}\sum_{i=h}^{k}\binom{a_i-s}{i-j}
    -\left( \sum_{i=h}^{k}\binom{a_i-1}{i-1} - \sum_{j=0}^{t-2}\binom{s-1}{j}\sum_{i=h}^{k}\binom{a_i-s}{i-1-j} \right) \notag\\
& = \sum_{i=h}^{k}\binom{a_i}{i} - \sum_{i=h}^{k}\binom{a_i-1}{i-1}
    - \sum_{j=0}^{t-1}\left(\binom{s}{j}-\binom{s-1}{j-1}\right)\sum_{i=h}^{k}\binom{a_i-s}{i-j} \notag\\
& = \sum_{i=h}^{k}\binom{a_i-1}{i} - \sum_{j=0}^{t-1}\binom{s-1}{j}\sum_{i=h}^{k}\binom{a_i-s}{i-j}.  \notag
\end{align}
Since $\mathcal{H}(\bar{1}) \subset EM(n,k,s-1,t)$, by the induction hypothesis
\begin{align}
|\partial\mathcal{H}(\bar{1})|
>  \sum_{i=h}^{k}\binom{a_i-1}{i-1} - \sum_{j=0}^{t-2} \binom{s-1}{j}\sum_{i=h}^{k}\binom{a_i-s}{i-1-j}
> |\mathcal{H}(1)|,  \notag
\end{align}
which contradicts the assumption that $\mathcal{H}$ is shifted.
\end{proof}

Since $\mathcal{H}(1) \subset EM(n,k-1,s-1,t-1)$, by the induction hypothesis and Claim \ref{claim-H(1)-k-s-t},
\begin{align}
|\partial\mathcal{H}|
&  \ge |\mathcal{H}(1)| + |\partial\mathcal{H}(1)| \notag\\
& \ge  \sum_{i=h}^{k}\binom{a_i-1}{i-1} - \sum_{j=0}^{t-2}\binom{s-1}{j}\sum_{i=h}^{k}\binom{a_i-s}{i-1-j}  \notag\\
& \quad   + \sum_{i=h}^{k}\binom{a_i-1}{i-2} - \sum_{j=0}^{t-3}\binom{s-1}{j}\sum_{i=h}^{k}\binom{a_i-s}{i-2-j}  \notag \\
& = \sum_{i=h}^{k}\binom{a_i}{i-1} - \sum_{j=0}^{t-2}\left(\binom{s-1}{j}+\binom{s-1}{j-1}\right)\sum_{i=h}^{k}\binom{a_i-s}{i-1-j} \notag\\
& = \sum_{i=h}^{k}\binom{a_i}{i-1} - \sum_{j=0}^{t-2} \binom{s}{j}\sum_{i=h}^{k}\binom{a_i-s}{i-1-j}. \notag
\end{align}
This completes the proof of Lemma \ref{lemma-shadow-k-s-t-exact}.
\end{proof}

\begin{coro}\label{coro-ell-shadow-k-s-t-exact}
Let $s\ge t \ge 0$ and $1\le \ell \le k-1$.
Suppose that $\mathcal{H} \subset EM(n,k,s,t)$ and $|\mathcal{H}| = m$.
Then $|\partial_{\ell}\mathcal{H}|\ge |\partial_{\ell} L_{m}EM(n,k,s,t)|$.
%and equality holds iff $\mathcal{H} \cong L_{m}EM(n,k,s,t)$.
%In particular, if $|\mathcal{H}|=\binom{x}{k} - \sum_{j=0}^{t-1}\left(\binom{s}{j}\binom{x-s}{k-j}\right)$
%for some $x\in \mathbb{R}$, then $|\partial\mathcal{H}| \ge \binom{x}{k-\ell} - \sum_{j=0}^{t-\ell-1}\left(\binom{s}{j}\binom{x-s}{k-\ell-j}\right)$.
\end{coro}
\begin{proof}
Similar to the proof of Lemma \ref{lemma-shadow-k-s-t-exact},
it suffices to show that if for some integers $a_k> \cdots > a_h\ge h\ge \max\{t,1\}$
\begin{align}
|\mathcal{H}| = \sum_{i=h}^{k}\binom{a_i}{i} - \sum_{j=0}^{t-1} \binom{s}{j}\sum_{i=h}^{k}\binom{a_i-s}{i-j}, \notag
\end{align}
then
\begin{align}
|\partial_{\ell}\mathcal{H}|
\ge \sum_{i=h}^{k}\binom{a_i}{i-\ell} - \sum_{j=0}^{t-1-\ell} \binom{s}{j}\sum_{i=h}^{k}\binom{a_i-s}{i-\ell-j}. \notag
\end{align}

We proceed by induction on $\ell$.
When $\ell = 1$, this is Lemma \ref{lemma-shadow-k-s-t-exact}.
So we may assume that $\ell \ge 2$.
By the induction hypothesis
\begin{align}
|\partial_{\ell-1}\mathcal{H}|
\ge \sum_{i=h}^{k}\binom{a_i}{i-\ell+1} - \sum_{j=0}^{t-\ell}\binom{s}{j}\sum_{i=h}^{k}\binom{a_i-s}{i-\ell+1-j}. \notag
\end{align}
Since $\partial_{\ell-1}\mathcal{H} \subset EM(n,k,s,t-\ell+1)$,
by Lemma \ref{lemma-shadow-k-s-t-exact},
\begin{align}
|\partial_{\ell}\mathcal{H}| = |\partial \partial_{\ell-1}\mathcal{H}|
\ge \sum_{i=h}^{k}\binom{a_i}{i-\ell} - \sum_{j=0}^{t-\ell-1}\binom{s}{j}\sum_{i=h}^{k}\binom{a_i-s}{i-\ell-j}. \notag
\end{align}
This completes the proof of Corollary \ref{coro-ell-shadow-k-s-t-exact}.
\end{proof}

The same induction argument as above gives the following technically simpler version of
Corollary \ref{coro-ell-shadow-k-s-t-exact}.

\begin{lemma}[Simplified version of Lemma \ref{lemma-shadow-k-s-t-exact}]\label{lemma-shadow-k-s-t-simple}
Let $s\ge t \ge 0$ and $1\le \ell \le k-1$.
Suppose that $\mathcal{H} \subset EM(n,k,s,t)$ and
$|\mathcal{H}| = \binom{x}{k} - \sum_{j=0}^{t-1}\binom{s}{j}\binom{x-s}{k-j}$ for some $x\in \mathbb{R}$.
Then $|\partial_{\ell} \mathcal{H}|\ge \binom{x}{k-\ell} - \sum_{j=0}^{t-1-\ell}\binom{s}{j}\binom{x-s}{k-\ell-j}$.
\end{lemma}
%%%%%%%%%%%%%%%%%%%%%%%%%%%%%%%%%%%%%%%%%

Let
\begin{align}
HM(n,k,s,t) = & \left\{ A\in\binom{[n]}{k}: |A\cap [s-1]|\ge 1\right\} \cup  \notag\\
               & \left\{A\in\binom{[n]}{k}: s\in A \text{ and } |A\cap [s+1,s+t]|\ge 1\right\}. \notag
\end{align}
Note that there is no constraint on the relation between $s$ and $t$ for $HM(n,k,s,t)$.

Similar to Lemmas \ref{lemma-decompose-LmEM(n,k,s,t}, \ref{lemma-k-cascade-resp}, and \ref{lemma-size-shadow-LmEM(n,k,s,t)}
we have the following result for $HM(n,k,s,t)$.

\begin{lemma}\label{lemma-decompose-HM(n,k,s,t)}
Let $n\ge k$. Then the following hold.
\begin{itemize}
\item[(a)] $|HM(n,k,s,t)| = \binom{n}{k} - \binom{n-s}{k} - \binom{n-s-t}{k-1}$ and $|\partial HM(n,k,s,t)| = \binom{n}{k-1}$.
\item[(b)] For every $m \le |HM(n,k,s,t)|$ there exist integers $a_{k}> \cdots > a_{h} \ge h \ge 1$ such that
\begin{align}
L_{m}HM(n,k,s,t) = HM(a_{k},k,s,t) \cup \bigcup_{i=h}^{k-1}
\left(HM(a_{i},i,s,t)+\{\widehat{a}_{i+1}, \ldots, \widehat{a}_{k}\}\right). \notag
\end{align}
\item[(c)] For every $m\ge 1$ there exists a unique sequence of integers $a_{k}> \cdots > a_{h} \ge h \ge 1$ such that
\begin{align}
m = \sum_{i=h}^{k}\binom{a_i}{i} - \sum_{i=h}^{k}\binom{a_{i}-s}{i} - \sum_{i=h}^{k}\binom{a_i-s-t}{i-1}. \notag
\end{align}
\item[(d)] If $m$ is given by the equation above, then
\begin{align}
|\partial L_{m}HM(n,k,s,t)| = \sum_{i=h}^{k}\binom{a_i}{i-1}. \notag
\end{align}
\end{itemize}
\end{lemma}

\begin{lemma}\label{lemma-shadow-HM(n,k,s,t)-exact}
If $\mathcal{H} \subset HM(n,k,s,t)$ and $|\mathcal{H}| = m$, then $|\partial\mathcal{H}|\ge |\partial L_{m}HM(n,k,s,t)|$.
In particular, if $|\mathcal{H}| = \sum_{i=h}^{k}\binom{a_i}{i} - \sum_{i=h}^{k}\binom{a_{i}-s}{i} - \sum_{i=h}^{k}\binom{a_i-s-t}{i-1}$
for some integers $a_{k}> \cdots > a_{h} \ge h \ge 1$,
then $|\partial\mathcal{H}| \ge \sum_{i=h}^{k}\binom{a_i}{i-1}$.
%In particular, if
%$|\mathcal{H}| = \binom{x}{k} - \binom{x-s}{k} - \binom{x-s-t}{k-1}$ for some $x\in \mathbb{R}$.
%Then $|\partial\mathcal{H}| \ge \binom{x}{k-1}$.
\end{lemma}
\begin{proof}
Let $a_{k}> \cdots > a_{h} \ge h \ge 1$ be integers such that
\begin{align}
m = \sum_{i=h}^{k}\binom{a_i}{i} - \sum_{i=h}^{k}\binom{a_{i}-s}{i} - \sum_{i=h}^{k}\binom{a_i-s-t}{i-1}.\notag
\end{align}
%where $a_k> \cdots > a_h\ge h\ge \max\{t,1\}$.
%Moreover, it is easy to see that
%\begin{align}
%|L_{m}HM(n,k,s,t)| = \sum_{i=h}^{k}\binom{a_i}{i} - \sum_{i=h}^{k}\binom{a_{i}-s}{i} - \sum_{i=h}^{k}\binom{a_i-s-t}{i-1}, \notag
%\end{align}
%and
Then by Lemma \ref{lemma-decompose-HM(n,k,s,t)}, it suffices to show
%\begin{align}
%|\partial L_{m}HM(n,k,s,t)| = \sum_{i=h}^{k}\binom{a_i}{i-1}. \notag
%\end{align}
%Therefore, it suffice to show that
%if for some integers $a_k> \cdots > a_h\ge h\ge \max\{t,1\}$
%\begin{align}
%|\mathcal{H}| = \sum_{i=h}^{k}\binom{a_i}{i} - \sum_{i=h}^{k}\binom{a_{i}-s}{i} - \sum_{i=h}^{k}\binom{a_i-s-t}{i-1}, \notag
%\end{align}
%then
\begin{align}
|\partial\mathcal{H}| \ge  \sum_{i=h}^{k}\binom{a_i}{i-1}=|\partial L_{m}HM(n,k,s,t)|. \notag
\end{align}

We proceed by induction on $s$ and $t$.
When $s = 0$, this is trivially true.
When $t=0$, we have $HM(n,k,s,0) = EM(n,k,s,1)$, so the conclusion follows from
Lemma \ref{lemma-shadow-k-s-t-exact}.
So we may assume that $s\ge 1$ and $t\ge 1$.

\begin{claim}\label{claim-size-H(1)-in-HM(n,k,s,t)}
$|\mathcal{H}(1)| \ge \sum_{i=h}^{k}\binom{a_i-1}{i-1}$
\end{claim}
\begin{proof}[Proof of Claim \ref{claim-size-H(1)-in-HM(n,k,s,t)}]
Suppose not. Then
\begin{align}
|\mathcal{H}(\bar{1})| = |\mathcal{H}| - |\mathcal{H}(1)|
                 > \sum_{i=h}^{k}\binom{a_i-1}{i} - \sum_{i=h}^{k}\binom{a_i-s}{i} - \sum_{i=h}^{k}\binom{a_i-s-t}{i-1}.\notag
\end{align}
Since $\mathcal{H}(\bar{1})\subset HM(n,k,s-1,t)$,
by the induction hypothesis
$|\partial\mathcal{H}(\bar{1})|> \sum_{i=h}^{k}\binom{a_i-1}{i-1} > |\mathcal{H}(1)|$,
which contradicts the assumption that $\mathcal{H}$ is shifted.
\end{proof}

Now, by Claim \ref{claim-size-H(1)-in-HM(n,k,s,t)} and the Kruskal-Katona theorem,
\begin{align}
|\partial\mathcal{H}| \ge |\mathcal{H}(1)| + |\partial\mathcal{H}(1)|
                      \ge \sum_{i=h}^{k}\binom{a_i-1}{i-1} + \sum_{i=h}^{k}\binom{a_i-1}{i-2}
                      = \sum_{i=h}^{k}\binom{a_i}{i-1}. \notag
\end{align}
This completes the proof of Lemma \ref{lemma-shadow-HM(n,k,s,t)-exact}.
\end{proof}
%%%%%%%%%%%%%%%%%%%%%%%%%%%%%%%%%%%%%%%%%%%%%

Similarly, the same induction argument as above gives the following technically simpler version of Lemma \ref{lemma-shadow-HM(n,k,s,t)-exact}.

\begin{lemma}[Simplified version of Lemma \ref{lemma-shadow-HM(n,k,s,t)-exact}]\label{lemma-shadow-HM(n,k,s,t)-simple}
Suppose that $\mathcal{H} \subset HM(n,k,s,t)$ and
$|\mathcal{H}| = \binom{x}{k} - \binom{x-s}{k} - \binom{x-s-t}{k-1}$ for some $x\in \mathbb{R}$.
Then $|\partial\mathcal{H}| \ge \binom{x}{k-1}$.
\end{lemma}
%%%%%%%%%%%%%%%%%%%%%%%%%%%%%%%%%%%%%%%%%%%%%%%%%%

\subsection{Proof of Theorem \ref{thm-shadow-intersecting-all-n}}
The proof of Theorem \ref{thm-shadow-intersecting-all-n} uses the following structural theorem for intersecting families.

For $1\le t\le k-1$ let
\begin{align}
    HM(n,k,t) = & \left\{A\in EKR(n,k): A \cap [2,k] \neq \emptyset \text{ or } [k+1,k+t] \subset A \right\} \cup \notag\\
    & \bigcup_{i=1}^{t} \{\{2,\ldots,k,k+i\}\}. \notag
\end{align}
Note that $|HM(n,k,2)| = \binom{n-1}{k-1}-\binom{n-k}{k-1}+\binom{n-k-2}{k-3}+2$ and $HM(n,k,1)$ is the extremal configuration in the Hilton-Milner theorem on nontrivial intersecting families.

\begin{thm}[Han and Kohayakawa, \cite{HK17}]\label{thm-HK17-structure-HM'}
Let $k\ge 3$ and $n> 2k$ and let $\mathcal{H}$ be an $n$-vertex intersecting $k$-graph.
If $\mathcal{H} \not\subset EKR(n,k)$ and $\mathcal{H} \not\subset HM(n,k,1)$ and for $k=3$ $\mathcal{H} \not\subset EM(n,3,3,2)$ as well,
then $|\mathcal{H}| \le |HM(n,k,2)|$.
\end{thm}

\begin{proof}[Proof of Theorem \ref{thm-shadow-intersecting-all-n}]
By the assumption on the size of $\mathcal{H}$ and Theorem \ref{thm-HK17-structure-HM'},
for $k\ge 4$ either $\mathcal{H} \subset EKR(n,k)$ or $\mathcal{H} \subset HM(n,k,1)$,
and for $k=3$ we have $\mathcal{H} \subset EKR(n,3)$.

Suppose that $k=3$.
Since $\mathcal{H} \subset EKR(n,3) = EM(n,3,1,1)$, by Corollary \ref{coro-ell-shadow-k-s-t-exact},
$|\partial_{\ell}\mathcal{H}| \ge |\partial_{\ell} L_{m}EKR(n,3)|$ for $1\le \ell \le 2$.
%and equality holds iff $\mathcal{H} \cong L_{m}EM(n,3)$.

Now suppose that $k\ge 4$.
Let $a_k> \cdots > a_h \ge h \ge 1$ be integers such that $|\mathcal{H}| = \sum_{i=h}^{k}\binom{a_i}{i} - \sum_{i=h}^{k}\binom{a_i-1}{i}$.
If $\mathcal{H} \subset EKR(n,k)=EM(n,k,1,1)$, then by Corollary \ref{coro-ell-shadow-k-s-t-exact},
$|\partial_{\ell}\mathcal{H}|\ge  |\partial_{\ell}L_{m}EM(n,k,1,1)|
=|\partial_{\ell}L_{m}EKR(n,k)|$ and we are done.
So we may assume that $\mathcal{H} \subset HM(n,k,1)$ and we are going to show that
% $|\partial_{\ell}\mathcal{H}|> \sum_{i=h}^{k}\binom{a_i}{i-1} - \sum_{i=h}^{k}\binom{a_i-1}{i} = \sum_{i=h}^{k}\binom{a_i}{i-\ell}$ in this case.
$|\partial_{\ell}\mathcal{H}|> \sum_{i=h}^{k}\binom{a_i}{i-\ell}=|\partial_{\ell}L_{m}EKR(n,k)|$ in this case.

Suppose that $|\partial_{\ell}\mathcal{H}| \le \sum_{i=h}^{k}\binom{a_i}{i-\ell}$.
Let $\mathcal{H}' = \mathcal{H}\setminus\{\{2,\ldots,k+1\}\}$
and note that $|\partial_{\ell}\mathcal{H}'|\le|\partial_{\ell}\mathcal{H}| \le \sum_{i=h}^{k}\binom{a_i}{i-\ell}$.
Applying the the contrapositive of the Kruskal-Katona theorem to $\mathcal{H}'$ we obtain $|\partial \mathcal{H}'| \le \sum_{i=h}^{k}\binom{a_i}{i-1}$.
On the other hand, since $\mathcal{H} \subset HM(n,k,1)$, $\mathcal{H}' \subset HM(n,k,1)\setminus \{\{2,\ldots,k+1\}\} = HM(n,k,1,k)$.
So applying the contrapositive of Lemma \ref{lemma-shadow-HM(n,k,s,t)-exact}
to $\mathcal{H}'$ we obtain
\begin{align}
|\mathcal{H}|
\le |\mathcal{H}'|+1
& \le \sum_{i=h}^{k}\binom{a_i}{i} - \sum_{i=h}^{k}\binom{a_i-1}{i} - \sum_{i=h}^{k}\binom{a_i-1-k}{i-1} + 1 \notag\\
& = \sum_{i=h}^{k}\binom{a_i-1}{i-1} - \sum_{i=h}^{k}\binom{a_i-1-k}{i-1} + 1. \notag
\end{align}

\begin{claim}\label{claim-size-a_k-a_k-1}
$a_k \ge 2k$ and if $a_k = 2k$ then $a_{k-1} = 2k-1$.
\end{claim}
\begin{proof}[Proof of Claim \ref{claim-size-a_k-a_k-1}]
First, suppose that $a_k \le 2k-1$. Then
\begin{align}
|\mathcal{H}|   & \le \sum_{i=h}^{k}\binom{a_i-1}{i-1} - \sum_{i=h}^{k}\binom{a_i-1-k}{i-1} + 1 \notag\\
                & \le \sum_{i=1}^{k}\binom{k+i-2}{i-1} - \sum_{i=1}^{k}\binom{i-2}{i-1} + 1 \notag\\
                &  =   \binom{2k-1}{k-1} + 1
                   <  \binom{2k}{k-1}-\binom{k}{k-1}-\binom{k-1}{k-2} +3, \notag
\end{align}
which contradicts the assumption that  $|\mathcal{H}| > |HM(n,k,2)|$ and $n> 2k$.
Therefore, $a_k \ge 2k$.

Now suppose that $a_k = 2k$ and $a_{k-1}\le 2k-2$.
Then, %for $k=4$,
\begin{align}
|\mathcal{H}|   & \le \sum_{i=h}^{k}\binom{a_i-1}{i-1} - \sum_{i=h}^{k}\binom{a_i-1-k}{i-1} + 1 \notag\\
                & \le \binom{2k-1}{k-1} + \sum_{i=1}^{k-1}\binom{k+i-2}{i-1} - \sum_{i=1}^{k}\binom{i-2}{i-1} \notag\\
                &  =   \binom{2k-1}{k-1} + \binom{2k-2}{k-2}
                   <  \binom{2k}{k-1}-\binom{k}{k-1}-\binom{k-1}{k-2} +3, \notag
\end{align}
where the strict inequality uses $k \ge 4$ and $n>2k$. This contradicts the assumption that  $|\mathcal{H}| > |HM(n,k,2)|$.
%For $k=3$, $|\mathcal{H}| \le \binom{5}{2}+\binom{3}{1}+\binom{2}{0}-\binom{2}{2}+1 < 15$, which contradicts the assumption that $|\mathcal{H}|\ge 15$.
%Therefore, $a_{k-1} = 2k-1$. This completes the proof of Claim \ref{claim-size-a_k-a_k-1}.
\end{proof}

Claim \ref{claim-size-a_k-a_k-1} implies that $\sum_{i=h}^{k}\binom{a_i-1-k}{i-1} - 1 > 0$.
Therefore, $$|\mathcal{H}| \le \sum_{i=h}^{k}\binom{a_i-1}{i-1} - \sum_{i=h}^{k}\binom{a_i-1-k}{i-1} + 1 < \sum_{i=h}^{k}\binom{a_i-1}{i-1},$$
contradicts the assumption that $|\mathcal{H}| = \sum_{i=h}^{k}\binom{a_i}{i} - \sum_{i=h}^{k}\binom{a_i-1}{i} = \sum_{i=h}^{k}\binom{a_i-1}{i-1}$.
Therefore, if $\mathcal{H} \subset HM(n,k,1)$, then $|\partial_{\ell}\mathcal{H}|> \sum_{i=h}^{k}\binom{a_i}{i-\ell}$,
and this completes the proof of Theorem \ref{thm-shadow-intersecting-all-n}.
\end{proof}
%%%%%%%%%%%%%%%%%%%%%%%%%%%%%%%%%%%%%%%%

\subsection{Proof of Theorem \ref{thm-shadow-t-intersecting-all-n}}
In this section we prove Theorem \ref{thm-shadow-t-intersecting-all-n}. We need the following  theorem for $t$-intersecting families.

\begin{thm}[Ahlswede and Khachatrian, \cite{AK96}]\label{thm-nontrivial-t-intersecting-AK96}
Let $t\ge 1, k\ge 3$, and $n> (t+1)(k-t+1)$.
Suppose that $\mathcal{H} \subset \binom{[n]}{k}$ is a $t$-intersecting family
and
\begin{align}
|\mathcal{H}| = m >
\begin{cases}
\max\left\{|AK(n,k,t)|,|EM(n,k,t+2,t+1)|\right\}, & \text{ if } t < \frac{k-1}{2}, \\
|EM(n,k,t+2,t+1)|, & \text{ if } t \ge \frac{k-1}{2}.
\end{cases}\notag
\end{align}
Then $\mathcal{H} \subset EM(n,k,t,t)$.
\end{thm}

\begin{proof}[Proof of Theorem \ref{thm-shadow-t-intersecting-all-n}]
Suppose $\mathcal{H}$ is given as in Theorem~\ref{thm-shadow-t-intersecting-all-n}.
By Theorem \ref{thm-nontrivial-t-intersecting-AK96},
$\mathcal{H} \subset EM(n,k,t,t)$ and by Corollary \ref{coro-ell-shadow-k-s-t-exact}, we have $|\partial_{\ell}{\mathcal H}| \ge |\partial_{\ell}L_mEM(n,k,t,t)|$.

%Let $t \ge \frac{k-1}{2}$. We now  show that the theorem does not hold if the value of $m(n,k,t)$ is reduced.
We now show that the value of $m(n,k,t)$ in the theorem cannot be reduced for $t \ge \frac{k-1}{2}$
and is tight up to a constant multiplicative factor for $t < \frac{k-1}{2}$.
Indeed, our construction is $\mathcal{H} = EM(n,k,t+2, t+1)$ and hence it suffices to prove the following.

\begin{fact} \label{fact1.7}
Let $n$ be sufficiently large and $m=|EM(n,k,t+2,t+1)|$.
Then $$|\partial_{\ell} EM(n,k,t+2,t+1)| < |\partial_{\ell} L_{m}EM(n,k,t,t)| \quad \text{ for all } \quad 1 \le \ell \le t.$$
In particular, for $1 \le \ell \le t$
the lower bound $m(n,k,t)$ for $|\mathcal{H}|$ in Theorem \ref{thm-shadow-t-intersecting-all-n} cannot be reduced to be less than $|EM(n,k,t+2,t+1)| \sim (t+2)\binom{n}{k-t-1}$.
\end{fact}

Note that when $t < \frac{k-1}{2}$ Fact \ref{fact1.7} implies that the constant multiplicative factor
is at most $|AK(n,k,t)|/|EM(n,k,t+2,t+1)| \sim \frac{k-t+1}{t+2}$ which is independent of $n$.

Let $x\in \mathbb{R}$ such that $\binom{x-t}{k-t} = |EM(n,k,t+2,t+1)| = (t+2)\binom{n-t-2}{k-t-1}+\binom{n-t-2}{k-t-2}$,
then $x = \Theta(n^{\frac{k-t-1}{k-t}})$.
Applying Lemma \ref{lemma-size-shadow-LmEM(n,k,s,t)} to $EM(n,k,t,t)$, we obtain
\begin{align}
|\partial_{\ell} L_{m}EM(n,k,t,t)|
& \ge \sum_{i= t-\ell}^{k-\ell}\binom{t}{i}\binom{x-t}{k-\ell-i} \notag\\
& = \binom{t}{t-\ell}\binom{x-t}{k-t} + (1+o(1))\binom{t}{t-\ell+1}\binom{x-t}{k-t-1} \notag\\
& = (t+2)\binom{t}{t-\ell}\binom{n}{k-t-1} + \Theta(n^{\frac{(k-t-1)^2}{k-t}}), \notag
\end{align}
and
\begin{align}
|\partial_{\ell} EM(n,k,t+2,t+1)|
& = \sum_{i=t+1-\ell}^{k-\ell}\binom{t+2}{i}\binom{n-t+2}{k-\ell-i} \notag\\
& = \binom{t+2}{t+1-\ell}\binom{n}{k-t-1} + \Theta(n^{k-t-2}). \notag
\end{align}

If $\ell < t$, then
$$|\partial_{\ell} L_{m}EM(n,k,t,t)| \sim (t+2)\binom{t}{t-\ell}\binom{n}{k-t-1} $$
and
$$|\partial_{\ell} EM(n,k,t+2,t+1)| \sim \binom{t+2}{t+1-\ell}\binom{n}{k-t-1}.$$
Since $\binom{t+2}{t+1-\ell} < (t+2)\binom{t}{t-\ell}$ for $\ell < t$,
$$|\partial_{\ell} EM(n,k,t+2,t+1)| < |\partial_{\ell} L_{m}EM(n,k,t,t)|$$ for large $n$.

If $\ell = t$, then $$|\partial_{\ell} L_{m}EM(n,k,t,t)| \sim (t+2)\binom{n}{k-t-1} + \Theta(n^{\frac{(k-t-1)^2}{k-t}})$$
and $$|\partial_{\ell} EM(n,k,t+2,t+1)| \sim (t+2)\binom{n}{k-t-1} + \Theta(n^{k-t-2}).$$
Since $\frac{(k-t-1)^2}{k-t} > k-t-2$,
$$|\partial_{t+1} EM(n,k,t+2,t+1)| < |\partial_{t+1} L_{m}EM(n,k,t,t)|$$ for large $n$. Consequently, Fact~\ref{fact1.7} holds and the proof is complete.
\end{proof}

\subsection{Proof of Theorem \ref{thm-EM(n,k,s)-n^k-2}}
Before proving Theorem \ref{thm-EM(n,k,s)-n^k-2} we need some structure theorems for
a family with large size and a given matching number.

\begin{dfn}
Let $n\ge sk+1$, $k \ge 3$, and $s\ge 1$.
Let $v_0,\ldots, v_{s-1} \in [n]$ be distinct vertices, $T_1,\ldots, T_s \subset [n]$ be pairwise disjoint $k$-sets,
and $v_i \in T_i$ for all $1\le i \le s-1$, and $v_0 \not\in T_i$ for all $1\le i \le s$.
Let
\begin{align}
PF(n,k,s) & = \left\{ T_1,\ldots, T_s \right\} \cup \notag\\
            & \quad \left\{A\in \binom{[n]}{k}:  \exists 0\le i \le s-1 \text{ such that } x_i\in A \text{ and } |A\cap \bigcup_{j=i+1}^{s}T_i|\ge 1 \right\}. \notag
\end{align}
\end{dfn}
Notice that $|PF(n,k,s)| \sim k\binom{s+1}{2}\binom{n}{k-2}$.
%%%%%%%%%%%%%%%%%%%%%%%%%%%%%%%%%%%%%%%%%%%%%%%%

\begin{thm}[Kostochka and Mubayi, \cite{KM17}]\label{thm-KM17-induction-n^k-2}
For every $k\ge 3$, $s\ge t \ge 2$, there exists $n_0$ such that the following holds for all $n\ge n_0$.
Suppose that $\mathcal{H} \subset \binom{[n]}{k}$ satisfies $\nu(\mathcal{H}) = s$ and
\begin{align}
|\mathcal{H}| >
\begin{cases}
|EM(n,3,s-t,1)| + |EM(n-s+t,3,2s+1,2)| & \text{ if } k=3, \\
|EM(n,k,s-t,1)| + |PF(n-s+t, k, t)| & \text{ if } k \ge 4.
\end{cases}\notag
\end{align}
Then there exists $X \subset [n]$ with $|X| = s-t+1$ such that $\nu(H-X) = t-1$.
The bound on $|\mathcal{H}|$ is tight.
In particular, if
\begin{align}
|\mathcal{H}| >
\begin{cases}
|EM(n,3,2s+1,2)| & \text{ for } k=3, \\
|PF(n, k, s)| & \text{ for } k \ge 4,
\end{cases}\notag
\end{align}
then there exists $v \in [n]$ such that $\nu(H-v) = s-1$.
\end{thm}

Note that the proof of Theorem \ref{thm-KM17-induction-n^k-2} was not included in \cite{KM17},
but one can easily prove it using results in \cite{FR18} (Theorems 4.1 and 4.2 in \cite{FR18}).
%
%Letting $t=s$ in Theorem \ref{thm-KM17-induction-n^k-2} we get the following corollary.
%
%\begin{coro}\label{coro-KM17-induction-n^k-2}
%For every $k\ge 3$, $s\ge 2$, there exists $n_0$ such that the following holds for all $n\ge n_0$.
%Suppose that $\mathcal{H} \subset \binom{[n]}{k}$ satisfies $\nu(\mathcal{H}) = s$ and
%$|\mathcal{H}| > |EM(n,3,2s+1,2)|$ for $k=3$ or $|\mathcal{H}| > |PF(n, k, s)|$ for $k\ge 4$.
%Then there exists $v \in [n]$ such that $\nu(H-v) = s-1$.
%\end{coro}

We also need the following structure theorems for intersecting families.  Let
%\begin{itemize}
%\item $H_{0}^{3}(n) = \{A\subset [n]: |A|=3, |A\cap [3]|\ge 2\}$.
%\item $H_{1}^{3}(n) = \{A\subset [n]: |A|=3, 1\in A, |A\cap [2,4]|\ge 1\} \cup \{234\}$.
%\item $H_{2}^{3}(n) = \{A\subset [n]: |A|=3, 1\in A, |A\cap \{2,3\}|\ge 1\} \cup \{234,235,145\}$.
%\item $H_{3}^{3}(n) = \{A \subset [n]: |A|=3, \{1,2\}\in A\} \cup \{134,135,145,234,235,245\}$.
%\item $H_{4}^{3}(n) = \{A \subset [n]: |A|=3, \{1,2\}\in A\} \cup \{ 134, 156, 235, 236, 245, 246 \}$.
%\item $H_{5}^{3}(n) = \{A \subset [n]: |A|=3, \{1,2\}\in A\} \cup \{ 134, 156, 136, 235, 236, 246 \}$.
%\item For $k\ge 4$ and $0\le i \le 5$, $H_{i}^{k}(n)=\{A \subset [n]: |A|=k, \exists B\in H_{i}(n) \text{ such that } B\subset A\}$.
%\end{itemize}
%
\begin{itemize}
\item $H_{0}^{3}(n) = \left\{ A \in \binom{[n]}{3}: |A\cap [3]|\ge 2 \right\}$.
\item $H_{1}^{3}(n) = \left\{ A \in \binom{[n]}{3}: 1\in A \text{ and } |A\cap \{2,3,4\}|\ge 1 \right\} \cup \{234\}$.
\item $H_{2}^{3}(n) = \left\{ A \in \binom{[n]}{3}: 1\in A \text{ and } |A\cap \{2,3\}|\ge 1 \right\} \cup \{234,235,145\}$.
\item $H_{3}^{3}(n) = \left\{ A \in \binom{[n]}{3}: \{1,2\}\in A \right\} \cup \{134,135,145,234,235,245\}$.
\item $H_{4}^{3}(n) = \left\{ A \in \binom{[n]}{3}: \{1,2\}\in A \right\} \cup \{ 134, 156, 235, 236, 245, 246 \}$.
\item $H_{5}^{3}(n) = \left\{ A \in \binom{[n]}{3}: \{1,2\}\in A \right\} \cup \{ 134, 156, 136, 235, 236, 246 \}$.
\item For $k\ge 4$ and $0\le i \le 5$, $H_{i}^{k}(n)=\left\{ A \in \binom{[n]}{k}: \exists B\in H_{i}^{3}(n) \text{ such that } B\subset A \right\}$.
\end{itemize}

\begin{fact}\label{fact-size-Hi(n)}
The following holds for all $n\ge k \ge 3$.
\begin{itemize}
\item $|H_{0}^{k}(n)| = 3\binom{n-3}{k-2} + \binom{n-3}{k-3} < 3\binom{n}{k-2} - 2\binom{n}{k-3}$.
\item $|H_{1}^{k}(n)| = 3\binom{n-4}{k-2} + 4\binom{n-4}{k-3} + \binom{n-4}{k-4} < 3\binom{n}{k-2} - 2\binom{n}{k-3}$.
\item $\max\{|H_{i}^{k}(n)|: 2\le i \le 5\} \le 2 \binom{n}{k-2}$.
\end{itemize}
\end{fact}

\begin{dfn}
Let $n\ge 2k$ and $k \ge 3$.
Let $Y = [2,k+1], Z=[k+2,2k]$. The $n$-vertex $k$-graph $PF(n,k)$ consists of all $k$-subsets of $[n]$ containing a member
of the family
\begin{align}
G = & \left\{ A: 1\in A\text{ and } |A\cap Y| = 1\text{ and } |A\cap Z|= 1 \right\}  \cup \notag\\
    & \left\{ Y, \{1,k,k+1\}, Z\cup \{k\},Z\cup \{k+1\} \right\}.\notag
\end{align}
\end{dfn}
Note that $|PF(n,k)| = O(n^{k-3})$.

\begin{thm}[Kostochka and Mubayi, \cite{KM17}]\label{thm-MK17-structure-nv(H)=s-stability}
Let $k\ge 4$ be fixed and $n$ be sufficiently large.
Then there is $C>0$ such that for every intersecting $n$-vertex $k$-graph $\mathcal{H}$ with $|\mathcal{H}|>|PF(n,k)| = O(n^{k-3})$,
one can remove from $\mathcal{H}$ at most $Cn^{k-4}$ edges
so that the resulting $k$-graph $\mathcal{H}'$ is contained in one of
$H_{0}^{k}(n),\ldots,H_{5}^{k}(n),EKR(n,k)$.
\end{thm}

For intersecting $3$-graphs there is a stronger result.
Define
\begin{align}
\tau(\mathcal{H}) = \min\{|S|: S\subset V(\mathcal{H}) \text{ and } |S\cap A|\ge 1 \text{ for all } A\in \mathcal{H}\}. \notag
\end{align}

\begin{thm}[Kostochka and Mubayi, \cite{KM17}]\label{thm-MK17-structure-1}
Let $\mathcal{H}$ be an intersecting $3$-graph and $n = |V(\mathcal{H})| \ge 6$.
If $\tau(\mathcal{H})\le 2$, then $\mathcal{H}$ is contained in one of
$EKR(n,3),H_{1}^{3}(n),\ldots,H_{5}^{3}(n)$.
\end{thm}

The following result shows that the size of an intersecting $3$-graph $\mathcal{H}$ with $\tau(\mathcal{H})\ge 3$ is bounded by a constant.

\begin{thm}[Frankl, \cite{FR80}]\label{thm-Frankl-cover-number-3}
Let $k\ge 3$ and $n$ be sufficiently large.
Then every intersecting $n$-vertex $k$-graph $\mathcal{H}$ with $\tau(\mathcal{H})\ge 3$
satisfies $|\mathcal{H}|\le |PF(n,k)|$.
Moreover, if $k\ge 4$, then equality holds only if $\mathcal{H} \cong PF(n,k)$.
\end{thm}

%%%%%%%%%%%%%%%%%%%%%%%%%%%%%%%%%%%%%%%%%%%%%%%%%

Now we are ready to prove Theorem \ref{thm-EM(n,k,s)-n^k-2}.

\begin{proof}[Proof of Theorem \ref{thm-EM(n,k,s)-n^k-2}]
Let $n$ be sufficiently large and $c = c(k,s)$ be given by $(\ref{hlowerc})$.
We may assume that $\mathcal{H}$ is shifted and $\nu(\mathcal{H})=s$.
For every $v\in [n]$ let $d_{\mathcal{H}}(v) = |\{A\in \mathcal{H}: v\in A\}|$,
and let $\Delta = \max\{d_{\mathcal{H}}(v) : v\in [n]\}$.
%Notice that $|\partial L_{m}EM(n,k,s,1)| < m/s$ when $m > \binom{(s+1)k-1}{k}$.
%Suppose that $m = \sum_{i=h}^{k}\binom{a_i}{i} - \sum_{i=h}^{k}\binom{a_i-s}{i}$ for some integers $a_{k}> \cdots > a_{h} \ge h\ge 1$.
%Then by Lemma \ref{lemma-size-shadow-LmEM(n,k,s,t)}
%\begin{align}
%|\partial L_{m}EM(n,k,s,1)|
%= \sum_{i=h}^{k}\binom{a_i}{i-1}
%=\frac{1}{s} \sum_{i=h}^{k} s\binom{a_i}{i-1}
%\le \sum_{i=h}^{k} \frac{1}{s}\left( \binom{a_i}{i} - \binom{a_i-s}{i} \right)
%= \frac{m}{s}.\notag
%\end{align}
%Suppose that $m = \binom{x}{k} - \binom{x-s}{k}$.
%Then by Lemma \ref{lemma-size-shadow-LmEM(n,k,s,t)},
%\begin{align}
%|\partial L_{m}EM(n,k,s,1)| \le \binom{x}{k-1} \le (1+o(1))\frac{1}{s} \left( \binom{x}{k} - \binom{x-s}{k} \right).
%\end{align}
Suppose that $m = \sum_{i=h}^{k}\binom{a_i}{i} - \sum_{i=h}^{k}\binom{a_i-s}{i}$ for some integers $a_{k}> \cdots > a_{h} \ge h\ge 1$.
Then by Lemma \ref{lemma-size-shadow-LmEM(n,k,s,t)}, it suffice to show that
\begin{align}
|\partial\mathcal{H}| \ge \sum_{i=h}^{k}\binom{a_i}{i-1}
=|\partial L_{m}EM(n,k,s,1)|. \notag
\end{align}
Note that  $a_k \rightarrow \infty$ as $n \rightarrow \infty$
and so $m \sim s{a_k-1 \choose k-1}$.
\begin{claim}\label{claim-max-degree-H}
$\Delta \le \sum_{i=h}^{k} \binom{a_i-1}{i-1} = (1+o(1))\frac{m}{s}$.
\end{claim}
\begin{proof}[Proof of Claim \ref{claim-max-degree-H}]
Suppose that there exists $v\in [n]$ with $d_{\mathcal{H}}(v) > \sum_{i=h}^{k} \binom{a_i-1}{i-1}$.
Let $\mathcal{H}(v) = \{A \setminus \{v\} : v\in A \in \mathcal{H}\}$.
Then by the Kruskal-Katona theorem,
\begin{align}
|\partial\mathcal{H}| \ge |\mathcal{H}(v)| + |\partial\mathcal{H}(v)|
> \sum_{i=h}^{k} \binom{a_i-1}{i-1} +  \sum_{i=h}^{k} \binom{a_i-1}{i-2}
= \sum_{i=h}^{k} \binom{a_i}{i-1}, \notag
\end{align}
and we are done.
\end{proof}

We are going to use Theorem \ref{thm-KM17-induction-n^k-2} and Claim \ref{claim-max-degree-H}
to define a sequence of distinct vertices $v_1,\ldots, v_{s-1}$ and a sequence of $k$-graphs $\mathcal{H}_1,\ldots, \mathcal{H}_{s-1}$
such that $\nu(\mathcal{H}_i) = s-i$ and $|\mathcal{H}_i| > (1-o(1))\frac{s-i}{s} m$ for all $1\le i \le s-1$.
Since $\mathcal{H}$ is shifted, we may assume that $v_i = i$ for $1 \le i \le s-1$.

First, by the assumption on the size of $\mathcal{H}$ and Theorem \ref{thm-KM17-induction-n^k-2},
there exists $v_1 \in [n]$ such that $\mathcal{H}_1 := \mathcal{H}-v_1$ satisfies $\nu(\mathcal{H}_1) = s-1$.
By Claim \ref{claim-max-degree-H}, $d_{\mathcal{H}}(v_1) < (1+o(1))m/s$, so $|\mathcal{H}_1| > (1-o(1))\frac{s-1}{s}m$.

Now suppose that we have defined $\mathcal{H}_i$ for some $1\le i \le s-2$ such that
$\nu(\mathcal{H}_i) = s-i$ and $|\mathcal{H}_i| > (1-o(1))\frac{s-i}{s} m$.
Since
\begin{align}
|\mathcal{H}_i|
> (1-o(1)) \frac{s-i}{s} m
 \ge\frac{s-i}{s} c\binom{n}{k-2}  \ge
\begin{cases}
|EM(n,3,2(s-i)+1,2)|, & \text{ for } k=3, \\
|PF(n, k, s-i)|, & \text{ for } k \ge 4,
\end{cases}\notag
\end{align}
by Theorem \ref{thm-KM17-induction-n^k-2},
there exists $v_{i+1} \in [n]$ such that $\mathcal{H}_{i+1} := \mathcal{H}_{i}-v_{i+1}$ satisfies $\nu(\mathcal{H}_{i+1}) = s-i-1$.
By Claim \ref{claim-max-degree-H}, $|\mathcal{H}_{i+1}| > (1-o(1))\frac{s-i-1}{s}m$.

Note that $\mathcal{H}_{s-1}$ satisfies $\nu(\mathcal{H}_{s-1}) = 1$ and
\begin{align}
|\mathcal{H}_{s-1}| > (1-o(1))\frac{1}{s} m \ge \frac{1}{s} c \binom{n}{k-2} \ge 3\binom{n}{k-2} > |PF(n,k)|. \notag
\end{align}

If $k=3$, then by Theorem \ref{thm-Frankl-cover-number-3}, $\tau(\mathcal{H}_{s-1}) \le 2$.
Therefore, by Theorem \ref{thm-MK17-structure-1}, $\mathcal{H}$ is contained in one of $EKR(n,3),H_{1}^{3}(n),\ldots,H_{5}^{3}(n)$.
Since $|\mathcal{H}_{s-1}| > 3n$ and
by Fact \ref{fact-size-Hi(n)}, $\max_{0\le i \le 5}\{|H_{i}^{3}|\} \le  3n-8$,
we must have $\mathcal{H} \subset EKR(n,3) = EM(n,3,1,1)$.
Note that $\mathcal{H}_{s-1}$ is obtained from $\mathcal{H}$ by removing $s-1$ vertices, so $\mathcal{H} \subset EM(n,3,s,1)$.
Therefore, by Lemma \ref{lemma-shadow-k-s-t-exact}, $|\partial\mathcal{H}| \ge |\partial L_{m}EM(n,3,s,1)|$
and we are done.

Now we may assume that $k\ge 4$.
Then, by Theorem~\ref{thm-MK17-structure-nv(H)=s-stability}, one can remove at most $Cn^{k-4}$ edges from $\mathcal{H}_{s-1}$
such that the resulting $k$-graph $\mathcal{H}'$ is contained in one of $H_{0}^{k}(n),\ldots,H_{5}^{k}(n)$, $EKR(n,k)$.
Note that $|\mathcal{H}'| \ge |\mathcal{H}_{s-1}| - Cn^{k-4} > 3\binom{n}{k-2} - \binom{n}{k-3}$
and by Fact \ref{fact-size-Hi(n)}, $\max_{0\le i \le 5}\{|H_{i}^{k}|\} < 3\binom{n}{k-2} - 2\binom{n}{k-3}$,
so $\mathcal{H}' \subset EKR(n,k) = EM(n,k,1,1)$.
Here we need $n$ to be sufficient large so that $Cn^{k-4} < \binom{n}{k-3}$.

Note that $\mathcal{H}_{s-1}$ is obtained from $\mathcal{H}$ by removing $s-1$ vertices.
If $\mathcal{H}_{s-1} \subset EM(n,k,1,1)$, then $\mathcal{H} \subset EM(n,k,s,1)$ and by Lemma \ref{lemma-shadow-k-s-t-exact} we are done.
So we may assume that $\mathcal{H}_{s-1} \not\subset EM(n,k,1,1)$, i.e. $\mathcal{H}_{s-1} \setminus \mathcal{H}' \neq \emptyset$.
Let $A \in \mathcal{H}_{s-1} \setminus \mathcal{H}'$ and since $\mathcal{H}_{s-1}$ is shifted,
we may assume that $A = \{s+1,\ldots,s+k\}$.
Since $\mathcal{H}_{s-1}$ is intersecting, every edge in $\mathcal{H}'$ must have nonempty intersecting with $A$.
So $\mathcal{H}' \subset HM(n,k,1,k)$.
This implies that
%Note that $\mathcal{H}'$ is obtained from $\mathcal{H}$ by removing $s-1$ vertices and at most $Cn^{k-4}$ edges.
%Therefore,
one can remove at most $Cn^{k-4}$ edges from $\mathcal{H}$ such that the resulting $k$-graph $\mathcal{H}''$
satisfies $\mathcal{H}'' \subset HM(n,k,s,k)$.

%If $\mathcal{H} \subset EM(n,k,s,1)$, then by Lemma \ref{lemma-shadow-k-s-t-exact}, we are done.
%So we may assume that $\mathcal{H} \not\subset EM(n,k,s,1)$, i.e. $\mathcal{H} \setminus \mathcal{H}'' \neq \emptyset$.
%Let $A \in \mathcal{H} \setminus \mathcal{H}''$ and since $\mathcal{H}$ is shifted, we may assume that $A = \{s+1,\ldots,s+k\}$.

Let $y\in \mathbb{R}$ satisfy $|\mathcal{H}''| = \binom{y}{k} - \binom{y-s}{k}-\binom{y-s-k}{k-1}$.
Then by Lemma \ref{lemma-shadow-HM(n,k,s,t)-simple},
$|\partial\mathcal{H}|\ge |\partial\mathcal{H}''| \ge \binom{y}{k-1}$.
Let $x\in \mathbb{R}, a_{k}, \ldots, a_h \in \mathbb{N}$ such that $a_k > \cdots >a_h\ge h \ge 1$
and $|\mathcal{H}| = \binom{x}{k}-\binom{x-s}{k} = \sum_{i=h}^{k}\binom{a_i}{i} - \sum_{i=h}^{k}\binom{a_i-s}{i}$.
It is easy to see that $x \le a_{k}+1$.

\begin{claim}\label{claim-y-ge-x+1}
$y > x+1$.
\end{claim}
\begin{proof}[Proof of Claim \ref{claim-y-ge-x+1}]
Suppose not. Then
\begin{align}
& \binom{x+1}{k}-\binom{x+1-s}{k} -\binom{x+1-s-k}{k-1} \notag\\
\ge & \binom{y}{k}-\binom{y-s}{k} -\binom{y-s-k}{k-1}
\ge \binom{x}{k} - \binom{x-s}{k} - Cn^{k-4}. \notag
\end{align}
Since $|\mathcal{H}| \ge c\binom{n}{k-2}$, $x = \Omega(n^{\frac{k-2}{k-1}})$.
Therefore,
\begin{align}
& \binom{x+1}{k}-\binom{x+1-s}{k} -\binom{x+1-s-k}{k-1} \notag\\
= & \binom{x}{k} - \binom{x-s}{k} + \binom{x}{k-1}-\binom{x-s}{k-1} - \binom{x+1-s-k}{k-1} \notag\\
< & \binom{x}{k} - \binom{x-s}{k} - \frac{1}{2}\binom{x-s-k}{k-1} \notag\\
< & \binom{x}{k} - \binom{x-s}{k} - Cn^{k-4}, \notag
\end{align}
a contradiction. This completes the proof of Claim \ref{claim-y-ge-x+1}.
\end{proof}

By Claim \ref{claim-y-ge-x+1},
\begin{align}
|\partial\mathcal{H}|\ge \binom{y}{k-1} > \binom{x+1}{k-1} \ge \sum_{i=h}^{k}\binom{a_i}{i-1}. \notag
\end{align}
This completes the proof of Theorem \ref{thm-EM(n,k,s)-n^k-2}.
\end{proof}
%%%%%%%%%%%%%%%%%%%%%%%%%%%%%%%%%%%%%%%%%%%%%%%%%%%%%%%%%%%%%
%%%%%%%%%%%%%%%%%%%%%%%%%%%%%%%%%%%%%%%%%%%%%%%%%

\section{Concluding Remarks}
Let $\mathcal{H} \subset \binom{[n]}{3}$ be an intersecting family with $|\mathcal{H}| \ge PF(n,3) = 10$.
Then Theorems \ref{thm-MK17-structure-1} and \ref{thm-Frankl-cover-number-3}
completely determine the structure of $\mathcal{H}$.
One can use this structural result to determine the minimum size of $|\partial_{\ell} \mathcal{H}|$ completely for $1\le \ell \le 2$.
However, the calculation is very complicated and tedious, so we omit it here.

As we mentioned before, for $1 \le \ell \le t$ the lower bound for $|\mathcal{H}|$ in Theorem \ref{thm-shadow-t-intersecting-all-n}
above is tight for $t \ge \frac{k-1}{2}$ and can be improved for $t < \frac{k-1}{2}$.
Indeed, one can use the $\Delta$-system method (see \cite{KM17}) to prove the following result.

\begin{thm}\label{thm-t-intersecting-second-level}
Let $t\ge 1$, $k\ge 3$, $\epsilon>0$, and $n$ be sufficiently large.
Suppose that $\mathcal{H} \subset \binom{[n]}{k}$ is $t$-intersecting.
If $t < \frac{k-1}{2}$ and $|\mathcal{H}| > (k-t+\epsilon)\binom{n}{k-t-1}$,
then either $\mathcal{H} \subset AK(n,k,t)$ or $\mathcal{H} \subset EM(n,k,t,t)$.
If $t\ge \frac{k-1}{2}$ and $|\mathcal{H}| > (t+1+\epsilon)\binom{n}{k-t-1}$,
then $\mathcal{H}$ is contained in one of $AK(n,k,t), EM(n,k,t+2,t+1), EM(n,k,t,t)$.
\end{thm}

One can easily use Corollary \ref{coro-ell-shadow-k-s-t-exact} to show that for $1\le \ell \le t$,
$|\partial_{\ell} L_{m}AK(n,k,t)| > |\partial_{\ell} L_{m}EM(n,k,t,t)|$ for sufficiently large $n$ and $m$.
Therefore, by Theorem \ref{thm-t-intersecting-second-level}, we obtain the following result.

\begin{fact}
Let $k\ge 3$, $t <\frac{k-1}{2}$, $1\le \ell \le t$, $\epsilon>0$, and $n$ be sufficiently large.
Then every $t$-intersecting family $\mathcal{H} \subset \binom{[n]}{k}$ with $|\mathcal{H}| = m > (k-t+\epsilon)\binom{n}{k-t-1}$
satisfies $|\partial_{\ell}\mathcal{H}| \ge |\partial_{\ell} L_{m}EM(n,k,t,t)|$.
%and equality holds iff $\mathcal{H} \cong L_{m}EM(n,k,t,t)$.
\end{fact}
%%%%%%%%%%%%%%%%%%%%%%%%%%%%%%%%%%%%%%%%%%%%%%%%%
%\bibliographystyle{elsarticle-num}
\bibliographystyle{abbrv}
\bibliography{shadow_intersecting}
\end{document}